\documentclass[11pt]{article}

\usepackage{times,amssymb,amsmath,exscale,array,latexsym,pst-all}
\usepackage{graphicx}
\usepackage{epsfig}
\usepackage{epic,eepic}
\usepackage{color}
\definecolor{marin}{rgb}   {0.,   0.3,   0.7}
\definecolor{rouge}{rgb}   {0.8,   0.,   0.}
\definecolor{sepia}{rgb}   {0.8,   0.5,   0.}
\usepackage[colorlinks,citecolor=marin,linkcolor=rouge,
            bookmarksopen,
            bookmarksnumbered
           ]{hyperref}

\usepackage{epsf}

\newcommand{\R}{\mathbb{R}}

\newcommand{\C}{\mathbb{C}}
\newcommand{\T}{\mathbb{T}}

\newcommand{\e}{\ensuremath{\mathrm{e}}}

\newcommand{\cO}{{\mathcal O}}

\newcommand{\Arg}{\mathrm{Arg}}

\newtheorem{lemma}{Lemma}[section]

\newtheorem{theorem}[lemma]{Theorem}

\newtheorem{remark}[lemma]{Remark}

\numberwithin{equation}{section}

\newcommand{\QED}{\mbox{}\hfill \raisebox{-0.2pt}{\rule{5.6pt}{6pt}\rule{0pt}{0pt}}
          \medskip\par}

\newenvironment{proof}{\noindent{\bf Proof:}}{$\Box$}

\title{Optimized high-order splitting methods for some classes of parabolic equations}
\author{S. Blanes$^{1}$, F. Casas$^{2}$, P. Chartier$^{3}$ and A. Murua$^{4}$ \\[2ex]
$^{1}$ {\small\it Instituto de Matem\'atica Multidisciplinar, Universitat Polit\`ecnica de Val\`encia,  46022 Valencia, Spain}\\{
\small\it email: serblaza@imm.upv.es}\\[1ex]
$^{2}$ {\small\it Departament de Matem\`atiques and IMAC,
Universitat Jaume I, 12071 Castell\'on, Spain}\\{
\small\it email: Fernando.Casas@mat.uji.es}\\[1ex]
$^{3}$ {\small\it  INRIA Rennes and Ecole Normale Sup\'erieure de Cachan, Antenne de Bretagne,} \\
{\small\it Avenue Robert Schumann, 35170 Bruz, France.}\\
{\small\it email:  Philippe.Chartier@inria.fr}\\[1ex]
$^{4}$ {\small\it  EHU/UPV, Konputazio Zientziak eta A.A. saila, Informatika Fakultatea, 12071 Donostia/San Sebasti\'an, Spain}\\{
\small\it email: Ander.Murua@ehu.es}
}

\begin{document}
\mathsurround 0.8mm

\maketitle

\begin{abstract}
We are concerned with the numerical solution obtained by {\em
splitting methods}  of certain parabolic partial differential
equations. Splitting schemes of order higher than two with real
coefficients necessarily involve negative coefficients. It has
been demonstrated that this  second-order barrier can be overcome
by using splitting methods with {\em complex-valued}
coefficients (with positive real parts). In this way, methods of orders $3$ to $14$ by using
the Suzuki--Yoshida triple (and quadruple) jump composition
procedure have been explicitly built. Here we reconsider this technique and show
that it is inherently bounded to order $14$ and clearly sub-optimal
with respect to error constants. As an alternative, we solve
directly the algebraic equations arising from the order conditions
and construct methods of orders $6$ and $8$ that are the most
accurate ones available at present time, even when low accuracies
are desired. We also show that, in the general case, 14 is not an order barrier for splitting methods with complex coefficients with positive real part by
building explicitly a method of order $16$ as a composition of
methods of order 8.
\\[1ex]
{\bf Keywords}: composition methods, splitting methods, complex coefficients, parabolic evolution equations.
\\[1ex]
{\bf MSC numbers}: 65L05, 65P10, 37M15
\end{abstract}

\section{Introduction}
In this paper, we consider linear evolution equations
 of the form
\begin{eqnarray} \label{eq:para}
\frac{d u}{dt}(t)= A u(t) + Bu(t), \qquad u(0)=u_0,
\end{eqnarray}
where the (possibly unbounded) operators $A$, $B$ and $A+B$ generate $C^0$ semi-groups for positive $t$ over a finite or infinite Banach space $X$. Equations of this form are encountered in the context of {\em parabolic} partial differential equations, a prominent example being the inhomogeneous {\em heat equation}
\begin{eqnarray*}
\frac{\partial u}{\partial t}(x,t)= \Delta u(x,t) + V(x) u(x,t),
\end{eqnarray*}
where $t\geq 0$, $x \in \R^d$ or $x \in \T^d$ and $\Delta$ denotes the Laplacian in $x$.

A method of choice for solving numerically (\ref{eq:para}) consists in advancing the solution alternatively along  the exact (or numerical) solutions of the two problems
\begin{eqnarray*}
\frac{d u}{dt}(t)= A u(t) \qquad \mbox{ and } \qquad \frac{d u}{dt}(t)= B u(t).
\end{eqnarray*}
Upon using an appropriate sequence of steps, high-order approximations can be obtained -for instance with exact flows- as
\begin{eqnarray} \label{eq:splittingmethod}
\Psi(h) = \e^{h b_0 B} \, \e^{h a_1 A} \, \e^{h b_1 B} \, \cdots \, \e^{h a_s A} \, \e^{h b_s B}.
\end{eqnarray}
The simplest example within this class
is the {\em Lie-Trotter splitting}
\begin{equation}
\e^{h A} \, \e^{h B} \qquad \mbox{ or } \qquad \e^{h B} \, \e^{h A},
\end{equation}
which is a first order approximation to the solution of (\ref{eq:para}), while the {\em symmetrized} version
\begin{equation} \label{eq:strang2}
S(h)= \e^{h/2\, A} \, \e^{h B} \, \e^{h/2\, A} \qquad \mbox{ or } \qquad S(h)= \e^{h/2\, B} \, \e^{h A} \, \e^{h/2\, B}
\end{equation}
is referred to as {\em Strang splitting} and is an approximation of order $2$.

The application of splitting methods to evolutionary partial differential equations of parabolic or
mixed hyperbolic-parabolic type constitutes a very active
field of research.  For this class of problems it makes sense to split the spatial differential operator, each
part corresponding to a different physical contribution (e.g., reaction and diffusion). Although  the
formal analysis of splitting methods in this setting can be carried out by power series
 expansions (as in the case
of ordinary differential equations), several fundamental difficulties arise, however, when establishing
convergence and rigorous error bounds for unbounded operators \cite{holden10smf}. Partial results exist for hyperbolic problems \cite{tang95ebf,tang98caf,holden10smf}, parabolic problems
\cite{descombes02sff,hundsdorfer03nso}
and for the Schr\"odinger equation \cite{jahnke00ebf,lubich08osm}, even for high order
splitting methods \cite{thalhammer08hoe}.

In \cite{hansen08esf}, it has been established that, under the  two conditions stated below, a splitting method of the form (\ref{eq:splittingmethod})  is of order $p$ for problem (\ref{eq:para}) if and only if it is of order $p$ for {\em ordinary differential equations} in finite dimension. In other words, if and only if  the difference $\Psi(h)-\e^{h(A+B)}$  admits a formal expansion of the form
\begin{eqnarray}
\label{eq:formalLE}
  \Psi(h)- \e^{h(A+B)} = h^{p+1} E_{p+1} + h^{p+2} E_{p+2} + \cdots
\end{eqnarray}

The two referred conditions write (see \cite{hansen08esf} for a complete exposition):
\begin{enumerate}
\item {\em Semi-group property}: $A$, $B$ and $A+B$ generate $C^0$ semi-groups on $X$ and, for all positive $t$,
\begin{eqnarray*}
\|\e^{t A} \| \leq \e^{\omega_A t},  \qquad \|\e^{t B} \| \le \e^{\omega_B t} \quad \mbox{ and } \quad
\|\e^{t (A+B)} \| \le \e^{\omega t}
\end{eqnarray*}
for some positive constants $\omega_A$, $\omega_B$ and $\omega$.
\item {\em Smoothness property}: For any pair of multi-indices $(i_1,\ldots,i_m)$ and $(j_1,\ldots,j_m)$ with $i_1+ \cdots + i_m + j_1 + \cdots + j_m = p+1$, and for all $t \in [0,T]$,
\begin{eqnarray*}
\|A^{i_1} B^{j_1} \ldots A^{i_m} B^{j_m} \, \e^{t(A+B)}u_0\| \leq C
\end{eqnarray*}
for a positive constant $C$.
\end{enumerate}
However, designing high-order splitting methods for (\ref{eq:para}) is not as straightforward as it might seem at first glance. As a matter of fact, the operators $A$ and $B$ are only assumed to generate $C^0$ semi-groups  (and not groups). This means in particular that the flows $\e^{t A}$ and/or $\e^{t B}$ may not be defined for negative times (this is indeed the case, for instance, for
the Laplacian operator) and this prevents the use of methods which embed negative coefficients.  Given that splitting methods with real coefficients must have some of their coefficients $a_i$ and $b_i$ negative\footnote{The existence of at least one negative coefficient was shown in \cite{sheng89slp, suzuki91gto}, and the existence of a negative coefficient for both operators was proved in \cite{goldman96nos}. An elementary proof can be found in \cite{blanes05otn}. }
to achieve
order $3$ or more, this seems to indicate, as it has been believed for a long time within the numerical analysis community, that
it is only possible to apply exponential splitting methods of at most order $p = 2$. In order to circumvent this order-barrier, the papers \cite{hansen09hos} and \cite{castella09smw} simultaneously introduced {\em complex-valued} coefficients\footnote{Methods with complex-values coefficients have also been used in a similar context \cite{rosenbrock62sgi} or in celestial mechanics \cite{chambers03siw}.} with positive real parts. It can indeed be checked in many situations that the propagators $\e^{z A}$ and $\e^{z B}$ are still well-defined in a reasonable distribution sense for $z \in \C$, provided  that $\Re(z) \geq 0$.  Using this extension from the real line to the complex plane, the authors of  \cite{hansen09hos} and \cite{castella09smw} built up methods of orders $3$ to $14$ by considering a technique known as {\em triple-jump composition}\footnote{And its generalization to  {\em quadruple-jump}.} and  made popular by a series of  authors: Creutz \& Gocksch \cite{creutz89hhm}, Forest \cite{forest89cia}, Suzuki \cite{suzuki90fdo} and Yoshida \cite{yoshida90coh}. \\

In this work, we continue the search for new methods with complex coefficients with positive real parts.
Eventually, our objective is to show that, compared to the methods built in
\cite{hansen09hos} and \cite{castella09smw} by applying
the triple-jump (or quadruple-jump) procedure,  it is possible to
construct more efficient methods and also of higher order by solving directly the polynomial equations arising from the order conditions.
In particular, we construct methods of order $6$ and $8$ with minimal local error constants among the methods with minimal number of stages. We also construct a method of order $16$, obtained as a composition based on an appropriate 8th order splitting method.


Here we are particularly interested in obtaining new splitting methods for reaction-diffusion equations.
These constitute mathematical models that describe how the population of one or several species distributed in space evolves under the action of two concurrent phenomena: {\em reaction} between species in which predators eat preys and {\em diffusion}, which makes the species to spread out in space\footnote{Apart from biology and ecology, systems of this sort also appear in chemistry
(hence the term reaction), geology and physics.}. From a mathematical point of view,
they belong to the class of semi-linear parabolic partial differential equations
and can be represented in the general form
\begin{eqnarray*}
 \frac{\partial u}{\partial t}  = D \Delta u + F(u),
\end{eqnarray*}
where each component of the vector $u(x,t)\in\R^d$ represents the population of one species,
$D$ is the real diagonal matrix of diffusion coefficients and $F$ accounts for all local
interactions between species.\footnote{The choice $F(u) = u(1-u)$ yields Fisher's equation and is used to describe the spreading of biological populations; the choice $F(u) = u(1-u^2)$ describes Rayleigh--Benard convection;
the choice $F(u) = u(1-u)(u-\alpha)$ with  $0<\alpha<1$ arises in combustion theory and is referred to as Zeldovich equation.} Strictly speaking, the theoretical framework introduced in \cite{hansen09hos} does not cover this situation if $F$ is nonlinear, so that (apart from section \ref{sect:nt}, where we successfully integrate numerically an example with nonlinear $F$) we will think of $F$ as being linear. The important feature of $A=D \Delta$ here is that it has a  {\bf real} spectrum: hence, any splitting method involving complex steps with positive real part is suitable for that class of problems. In principle, one could even consider splitting methods with $a_i$'s having positive real part and unconstrained complex $b_i$'s.

It may be worth mentioning that the size of the arguments of the $a_i$ coefficients of the splitting method
is a critical factor when the diffusion operator involves
a complex number, for instance, an equation of the form
\begin{equation} \label{eq:landau}
\frac{\partial u}{\partial t} = \delta \Delta u + F(u),
\end{equation}
where $\delta$ is a complex number with $\Re(\delta) > 0$. The choice $F(u)=\mu_3 u^3 + \mu_2 u^2 + \mu_1 u + \mu_0$ is known as the cubic Ginzburg--Landau equation \cite{fauve88lsg}. In this situation, the values of the $\tilde{a}_i := \arg(\delta) + \arg(a_i)$ determine whether the splitting method makes  sense for this specific value of $\delta$. If for all $i=1, \ldots, s$, $\tilde{a}_i \in [-\frac{\pi}{2},+\frac{\pi}{2}]$ then the method is well defined.
In order for the method to be applicable to such class of equations, it would make sense trying to minimize the value of $\max_{i=1,\ldots,s} |\arg(a_i)|$.

In this work, however, we focus on the case where the operator $A$ has a real spectrum, and thus we will only require that $|\arg(a_i)|\leq \pi/2$ (i.e., $\Re(a_i)\geq 0$) while minimizing the local error coefficients. In this sense, we have observed that the coefficients of accurate splitting methods with  $\Re(a_i)\geq 0$  tend to have also $b_i$ coefficients with positive real parts.

The plan of the paper is the following.
In Section
\ref{sect:aob}, we shall prove that if an $s$-jump construction is
carried out from a basic symmetric second-order method, it is bounded
to order $14$ and no more and we will further justify why solving
directly the system of order conditions leads to more efficient
methods. In Section~\ref{sect:cmw}, we solve the corresponding
order conditions of methods based on compositions of second order
integrators and construct several splitting methods whose
coefficients have positive real part. In particular, in
subsection~\ref{ssect:comp} we present splitting methods of orders
$6$ and $8$, obtained as a composition scheme with Strang splitting as basic integrator.
In addition, with the aim of showing that 14 is
not an order barrier in general, we have built explicitly a method
of order $16$ as a composition of methods of order 8, which in turn is obtained by composing the second order Strang splitting. In section \ref{sect:nt} we describe the implementation of the various
methods obtained in this paper and show their efficiency as
compared to already available methods on two test problems.
Finally, section \ref{sec.6} contains some discussion and
concluding remarks.

\section{An order barrier for the $s$-jump construction} \label{sect:aob}

A simple and very fruitful technique to build high-order methods
is to consider compositions of low-order ones with fractional time
steps.  In this way, numerical integrators of arbitrarily high order can be obtained.
For splitting methods aimed to integrate problems of the form (\ref{eq:para}), it is necessary,
however, that the coefficients have positive real part. The procedure has been carried out in
\cite{castella09smw,hansen09hos}, where composition methods up to order 14 have been constructed. We shall
prove here that 14 constitutes indeed an order barrier for this kind of approach. In other words, the
composition technique used in \cite{castella09smw,hansen09hos} does not allow for the construction of methods having all their coefficients in $\C_+:=\{z \in \C: \Re(z) \geq 0\}$ with orders strictly greater than $14$.
With this goal in mind we consider the following two families of methods: \\

\

{\bf Family I.} Given a method of order $p$,  $\Phi^{[p]}(h) = \e^{h(A+B)}+\mathcal{O}(h^{p+1})$, a sequence of
methods of orders $p+1, p+2, \ldots$ can be defined recursively by the compositions
\begin{equation}\label{eq:familyI1}
\Phi^{[p+q]}(h)=\prod_{i=1}^{m_q} \Phi^{[p+q-1]}(\alpha_{q,i} h)
:=  \Phi^{[p+q-1]}(\alpha_{q,1} h) \, \cdots \ \Phi^{[p+q-1]}(\alpha_{q,m_{q}} h),
\end{equation}
where for all $q \geq 1$,
\begin{eqnarray*}
\left( \, \forall \; 1 \leq i \leq m_q, \; \alpha_{q,i} \neq 0 \right), \quad \sum_{i=1}^{m_q} \alpha_{q,i} = 1
\quad  \mbox{ and }  \quad   \sum_{i=1}^{m_q} \alpha_{q,i}^{p+q} = 0.
\end{eqnarray*}
(Hereafter, we will interpret the product symbol from left to right).
Notice that if $p+q$ is even, the second condition has only complex solutions. \\

\

{\bf Family II.} Given a {\bf symmetric} method of order $2p$,$\tilde{\Phi}^{[2p]}(h)$, a
sequence of methods of orders $2(p+1)$, $2(p+2), \ldots$ can be defined recursively by
the {\bf symmetric} compositions
\begin{equation}\label{eq:familyII1}
\forall \, q \geq 1,  \quad \tilde{\Phi}^{[2(p+q)]}(h)=\prod_{i=1}^{m_q} \tilde{\Phi}^{[2(p+q)-2]}(\alpha_{q,i} h)
\end{equation}
where $\alpha_{q,m_q+1-i}=\alpha_{q,i}, \ i=1,2,\ldots$, and for
all $q \geq 1$,
\begin{eqnarray*}
\left( \, \forall \; 1 \leq i \leq m_q, \; \alpha_{q,i} \neq 0 \right),  \quad  \sum_{i=1}^{m_q} \alpha_{q,i} = 1
\quad \mbox{ and }   \quad  \sum_{i=1}^{m_q} \alpha_{q,i}^{2(p+q)+1} = 0.
\end{eqnarray*}
Methods of this class with real coefficients have been constructed by Creutz \& Gocksch \cite{creutz89hhm}, Suzuki
\cite{suzuki90fdo} and Yoshida \cite{yoshida90coh}. However, the second condition clearly indicates that
at least one of the coefficients must be negative. In contrast, there exist many complex solutions with
coefficients in $\C_+$.

Generally speaking, starting from $\Phi^{[1]}(h)$, the $(p+1)$-th member of family I is of the form
\begin{equation}\label{eq:familyIIGen}
 \Phi^{[p+1]}(h)=  \prod_{i_p=1}^{m_p} \left( \prod_{i_{p-1}=1}^{m_{p-1}} \left( \ldots
  \left( \prod_{i_{1}=1}^{m_1} \Phi^{[1]}(\alpha_{p,i_{p}}\alpha_{p-1,i_{p-1}}
  \cdots\alpha_{1,i_{1}}h) \right) \ldots \right) \right)
\end{equation}
and has  coefficients
\begin{equation}\label{eq:familyIIProd}
  \prod_{j=1}^p\alpha_{j,i_j}  \qquad
   1\leq i_1 \leq m_1, \ldots, 1 \leq i_p \leq m_p.
\end{equation}

A similar expression holds, of course, for methods of family II,
starting from $\tilde{\Phi}^{[2]}(h)$. Symmetric compositions for
the cases $m_1=m_2=\cdots=m_p=3$ and $m_1=m_2=\cdots=m_p=4$,
correspond to the triple and quadruple jump techniques,
respectively.

\begin{lemma}\label{lemcons}
Let $\Phi(h)$ be a consistent method (i.e., a method of order $p\geq 1$) and assume that the method
\begin{equation}\label{lemma2.3.1}
  \Psi(h) = \prod_{i=1}^{l}\Phi(\alpha_i h)
\end{equation}
is also consistent (i.e., $\sum_{i} \alpha_i=1$).
 If  there exists $k$, $1\leq k \leq l$, such that  $\Re(\alpha_k)<0$, then any consistent method of the form
\begin{equation}
\prod_{j=1}^{m} \Psi(\beta_j h)
  =  \prod_{j=1}^{m} \left(
  \prod_{i=1}^{l} \Phi(\beta_j\alpha_ih)
  \right)
\end{equation}
has at least one coefficient $\beta_j \alpha_k$, $1\leq j \leq m$, such that  $\Re(\beta_j \alpha_k)<0$.
\end{lemma}

\begin{proof}
By consistency one has $\sum_{j=1}^m \beta_j=1$, so that
$$
\sum_{j=1}^m \Re(\beta_j \alpha_k)=  \Re \left( \sum_{j=1}^m \beta_j \alpha_k \right) = \Re \left(
   \alpha_k \sum_{j=1}^m \beta_j \right) = \Re(\alpha_k) <0.
$$
This implies  the statement.
\end{proof}

\begin{lemma}\label{lemsect}
For $k \geq 2$ and $r \geq 2$, consider $(z_1, \ldots, z_k) \in \left(\C_{+}\right)^k$ such that
$\sum_{i=1}^k z_i^r=0$, $z_i \ne 0$ for $i=1, \ldots, k$. Then we have
$$
\max_{i=1, \ldots,k} \Arg(z_i) - \min_{i=1, \ldots,k} \Arg(z_i) \geq  \frac{\pi}{r}.
$$
\end{lemma}

\begin{proof}
All the $z_i$'s belong to  the sector $K_{\sigma}(\theta) = \{ z \in \C: |\sigma-\Arg(z)| \leq \theta\}$ with
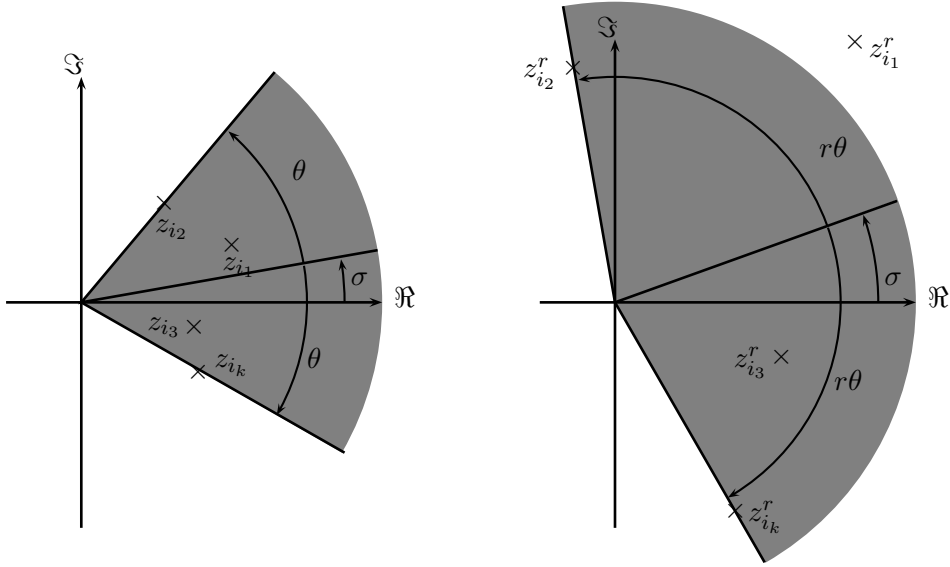
\begin{figure}
\begin{center}
\begin{pspicture}(-1,-3)(5,4)
\SpecialCoor
\pswedge*[linecolor=gray]{4}{10}{50}
\pswedge*[linecolor=gray]{4}{-30}{10}
\psline[linewidth=1pt](4;50)(0,0)(4;10)
\psline[linewidth=1pt](0,0)(3.5,-2.)
\psarc[arcsepB=2pt]{->}{3}{10}{50}
\rput[t](2.9,1.9){$\theta$}
\psarc[arcsepB=2pt]{->}{3.5}{0}{10}
\rput[t](3.7,0.4){$\sigma$}
\psarc[arcsepB=2pt]{<-}{3}{-30}{10}
\rput[t](3.1,-0.6){$\theta$}
\psline[linewidth=1pt]{->}(-1,0)(4,0)
\psline[linewidth=1pt]{->}(0,-3)(0,3)
\rput[t](-0.1,3.3){$\Im$}
\rput[t](4.3,0.2){$\Re$}
\rput[t](2.1,0.6){$z_{i_1}$}
\rput[t](1.2,1.1){$z_{i_2}$}
\rput[t](1.1,-0.2){$z_{i_3}$}
\rput[t](2.0,-0.7){$z_{i_k}$}
\rput[t](2.,0.9){$\times$}
\rput[t](1.1,1.45){$\times$}
\rput[t](1.5,-0.2){$\times$}
\rput[t](1.55,-0.8){$\times$}
\end{pspicture}
\begin{pspicture}(-2,-3)(5,4)
\SpecialCoor
\pswedge*[linecolor=gray]{4}{20}{100}
\pswedge*[linecolor=gray]{4}{-60}{20}
\psline[linewidth=1pt](0,0)(-0.69,3.94)
\psline[linewidth=1pt](0,0)(1.99,-3.46)
\psline[linewidth=1pt](0,0)(3.75,1.35)
\psarc[arcsepB=2pt]{->}{3}{20}{100}
\rput[t](2.9,2.2){$r \theta$}
\psarc[arcsepB=2pt]{->}{3.5}{0}{20}
\rput[t](3.7,0.4){$\sigma$}
\psarc[arcsepB=2pt]{<-}{3}{-60}{20}
\rput[t](3.1,-0.9){$r \theta$}
\psline[linewidth=1pt]{->}(-1,0)(4,0)
\psline[linewidth=1pt]{->}(0,-3)(0,3.5)
\rput[t](-0.1,3.8){$\Im$}
\rput[t](4.3,0.2){$\Re$}
\rput[t](3.6,3.52){$z_{i_1}^r$}
\rput[t](-1,3.25){$z_{i_2}^r$}
\rput[t](1.8,-0.6){$z_{i_3}^r$}
\rput[t](2.,-2.65){$z_{i_k}^r$}
\rput[t](3.19,3.6){$\times$}
\rput[t](-.56,3.25){$\times$}
\rput[t](2.21,-0.6){$\times$}
\rput[t](1.6,-2.65){$\times$}
\end{pspicture}
\end{center}
\caption{The enveloping sectors of $\{z_1, \ldots, z_k\} \subset \C_{+}$ and of $\{z_1^r, \ldots, z_k^r\} \subset \C$ (for $r=2$).\label{fig:envsect}}
\end{figure}
\begin{eqnarray*}
\sigma=\frac12 \left(\max_{i=1, \ldots,k} \Arg(z_i) + \min_{i=1, \ldots,k} \Arg(z_i) \right) \quad \mbox{ and } \quad
\theta = \frac12 \left(\max_{i=1, \ldots,k} \Arg(z_i) - \min_{i=1, \ldots,k} \Arg(z_i) \right),
\end{eqnarray*}
where $\Arg$ is the principal value of the argument (see the left picture of Figure \ref{fig:envsect}).
Now, assume that $\theta < \frac{\pi}{2 r}$. Then, all the $z_i^r$'s belong to $K_{r \sigma}(r \theta)$, which,
since $r \theta < \frac{\pi}{2}$, is a convex set. This implies that $\sum_{i=1}^k z_i^r$ also belongs to  $K_{r \sigma}(r \theta)$ (see the right picture of Figure \ref{fig:envsect}). The inequality $r \theta < \frac{\pi}{2}$ being strict and the $z_i$'s non-zero, we have furthermore $\sum_{i=1}^k z_i^r \neq 0$, which  contradicts the assumption. The result follows.
\end{proof}

\begin{theorem} \label{theorem2.1a}
(i) Starting from a first-order method $\Phi^{[1]}(h)$, all
methods $\Phi^{[p+1]}(h)$ of order $p+1$, $p=3, 4, \ldots$ from
Family I have at least one coefficient with negative real part.
(ii) Starting from a second-order symmetric method
$\tilde{\Phi}^{[2]}(h)$, all methods $\tilde{\Phi}^{[2p+2]}(h)$ of
order $2p+2$, $p=7, 8, \ldots$ from Family II have at least one
coefficient with negative real part.
\end{theorem}

\begin{proof}
We prove at once the two assertions. We first notice that, according to Lemma \ref{lemcons},
if method $\Phi^{[p]}(h)$ of Family I (respectively,
method $\tilde{\Phi}^{[2p]}(h)$ of Family II), has a coefficient with negative real part, then all subsequent
methods $\Phi^{[p+q]}(h)$, $q \geq 1$, of  Family I (respectively, methods $\tilde{\Phi}^{[2(p+q)]}(h)$ of Family II),
also have a coefficient with negative real part. Hence, we assume that all methods
$\Phi^{[q+1]}(h)$, $q=1, \ldots, p$ from Family I (respectively, all methods $\tilde{\Phi}^{[2q+2]}(h)$ of Family II),  have all their coefficients in $\C_{+}$.
Using Lemma \ref{lemsect} we  have
$$
\forall \, q=1, \ldots, p, \quad \max_{i=1, \ldots,m_q} \Arg(\alpha_{q,i}) - \min_{i=1, \ldots,m_q} \Arg(\alpha_{q,i}) \geq  \frac{\pi}{q+1}
$$
(respectively
$$
\forall q=1, \ldots, p, \quad \max_{i=1, \ldots,m_q} \Arg(\alpha_{q,i}) - \min_{i=1, \ldots,m_q} \Arg(\alpha_{q,i}) \geq  \frac{\pi}{2q+1}),
$$
so that
$$
\max_{1 \leq i_1 \leq m_1, \ldots, 1 \leq i_p \leq m_p} \Arg \left(\prod_{j=1}^p\alpha_{j,i_j} \right) - \min_{1 \leq i_1 \leq m_1, \ldots, 1 \leq i_p \leq m_p} \Arg \left(\prod_{j=1}^p\alpha_{j,i_j} \right) \geq  \frac{\pi}{2} + \cdots + \frac{\pi}{p+1}
$$
(respectively
$$
\max_{1 \leq i_1 \leq m_1, \ldots, 1 \leq i_p \leq m_p} \Arg \left(\prod_{j=1}^p\alpha_{j,i_j} \right) - \min_{1 \leq i_1 \leq m_1, \ldots, 1 \leq i_p \leq m_p} \Arg \left(\prod_{j=1}^p\alpha_{j,i_j} \right) \geq  \frac{\pi}{3} + \cdots + \frac{\pi}{2p+1}).
$$
Now,  since $\frac12+\frac13+\frac14>1$, $p=3$ in the first case and thus the first statement follows. For
Family II, since $\frac13+\frac15+\cdots+\frac{1}{15}>1$, then $p=7$, thus leading to the second statement.
 \end{proof}
 \begin{remark}
No method of Family I with coefficients in $\C_{+}$ can have an
order strictly greater than $3$. Such methods of order $3$ have
been constructed in \cite{hansen09hos}. Similarly, no method of
Family II with coefficients in $\C_{+}$ can have an order strictly
greater than $14$. Such methods with orders up to $14$ have been
constructed in \cite{castella09smw,hansen09hos}.
For instance, let us consider the quadruple jump composition
\begin{equation}\label{eq:Qjump}
 \tilde{\Phi}^{[2p+2]}(h)=
 \tilde{\Phi}^{[2p]}(\alpha_{p,1}h) \,
 \tilde{\Phi}^{[2p]}(\alpha_{p,2}h) \,
 \tilde{\Phi}^{[2p]}(\alpha_{p,2}h) \,
 \tilde{\Phi}^{[2p]}(\alpha_{p,1}h),
\end{equation}
where $\alpha_{p,1} +\alpha_{p,2} = 1/2$ and $\alpha_{p,1}^{2p+1}
+\alpha_{p,2}^{2p+1} = 0$. This system has $p$ solutions (and their
complex conjugate). The solution with minimal argument is, as noticed in
\cite{castella09smw,hansen09hos},
\begin{eqnarray*}
  \alpha_{p,1} = \frac14 \left( 1 + i
  \frac{\sin\left(\frac{\pi}{2p+1}\right)}{1+\cos\left(\frac{\pi}{2p+1}\right)} \right)
  , \qquad     \alpha_{p,2} =  \bar{\alpha}_{p,1}
\end{eqnarray*}
(and its complex conjugate). It is straightforward to verify that
$\arg(\alpha_{p,1})=\frac{\pi}{2(2p+1)}$ and
$\arg(\alpha_{p,2})=-\frac{\pi}{2(2p+1)}$, so that
\[
  \arg(\alpha_{p,1}) - \arg(\alpha_{p,2}) = \frac{\pi}{2p+1}.
\]
Comparing with the proof of
Theorem~\ref{theorem2.1a}, we observe that the bounds obtained there are
sharp since a family of methods do exist satisfying the equality.  
 \end{remark}
 \begin{remark}
It \emph{is} however possible to construct a composition method with all coefficients having positive real part of order strictly greater than $14$ directly from a symmetric second order method. For example, in Subsection~\ref{ssect:comp} we present a new method of sixteenth-order built as
\begin{equation}\label{eq:Ord16}
 \Phi^{[16]}(h)=
 \prod_{i=1}^{21} \Phi^{[8]}(\alpha_{i}h), \qquad \mbox{with}
 \qquad
 \Phi^{[8]}(h)=\prod_{j=1}^{15} \Phi^{[2]}(\beta_{j}h)
\end{equation}
and the coefficients satisfying $\Re(\alpha_i \beta_j) > 0$ for all $i=1,\ldots,21$, $j=1,\ldots,15$,
with $\alpha_{22-i}=\alpha_i, \ \beta_{16-j}=\beta_j, \
i,j=1,2,\ldots$. Here, $\Phi^{[8]}(h)$ is a symmetric
composition of symmetric second order methods, but it is not a
composition of methods of order 4 or 6, and similarly for
$\Phi^{[16]}(h)$, which is not a composition of methods of orders
$10$, $12$ or $14$.
\end{remark}

\begin{theorem} \label{theorem2.4}
 Splitting methods of the class (\ref{eq:splittingmethod}) with
$\Re(a_i)\geq 0$ exist at least up to order 44.
\end{theorem}

\begin{proof}
In \cite{castella09smw}, a fourth-order method was obtained with
$a_i\in\mathbb{R}^{+}$. In a similar way, we have also built a
sixth-order scheme with $a_i\in\mathbb{R}^{+}$ (whose coefficients
can be found at \newline
\texttt{http://www.gicas.uji.es/Research/splitting-complex.html}).
Using this as the basic method in the quadruple
jump (\ref{eq:Qjump}) with coefficients chosen with the minimal
argument and, since
\[
  \frac17+\cdots + \frac{1}{43}<1<\frac17+\cdots + \frac{1}{45}
\]
we conclude that all methods obtained up to order 44 will satisfy that
$\Re(a_i)\geq 0$.
\end{proof}

 \begin{remark}
The question of the existence of splitting schemes at any order
with $\Re(a_i)\geq 0$ remains still open.
\end{remark}

\section{Splitting methods with all coefficients having positive real part}
\label{sect:cmw}

\subsection{Order conditions and leading terms of local error}
\label{ssect:ocs}

We have seen that the composition technique to construct high
order methods inevitably leads to an order barrier. In addition,
the resulting methods require a large number of evaluations (i.e.
$3^{n-1}$ or $4^{n-1}$ evaluations to get order $2n$ using the
triple or quadruple jump, respectively) and usually have large
truncation errors. In this section we show that, as with real
coefficients, it is indeed possible to build very efficient high
order splitting methods whose coefficients have positive real part
by solving directly the order conditions necessary to achieve a
given order $p$. These are, roughly speaking, large systems of
polynomial equations in the coefficients $a_i$, $b_i$ of the
method (\ref{eq:splittingmethod}), arising by requiring that the
formal expansion of the method satisfies (\ref{eq:formalLE}) for
arbitrary non-commuting operators $A$ and $B$.

Different (but equivalent) formulations of such order conditions exist
in the literature \cite{hairer06gni}. Among them, the one using
Lyndon multi-indices is particularly appealing. It was first
introduced  in \cite{chartier09aat} (see also \cite{blanes08sac})
and can be considered as a variant of the classical treatment
obtained in \cite{murua99ocf}.

This analysis shows that the number of order conditions for general splitting methods of the form
(\ref{eq:splittingmethod}) grows very rapidly with the order $p$, even when one considers only symmetric methods.
For instance,
there are $26$ independent 8th-order conditions and $82$ 10th-order conditions for a
consistent symmetric splitting method. It makes sense, then, to examine alternatives to achieve order higher than six.  This can be accomplished by taking  compositions of a basic symmetric method of even order.
In particular, if we consider any of the two versions of Strang splitting (\ref{eq:strang2}) as the basic method $S(h)$, then, for each $\gamma = (\gamma_1,\ldots,\gamma_m) \in \C^m$,
\begin{eqnarray}
\label{eq:comppsi}
  \Psi(h) = S(\gamma_1 h) \cdots S(\gamma_s h)
\end{eqnarray}
will be  a new splitting method of the form (\ref{eq:splittingmethod}). Now the consistency condition reads
\begin{equation}
  \label{eq:consistencygamma}
  \gamma_{1}+\cdots + \gamma_{s}=1.
\end{equation}
As for the additional conditions to attain order $p$, these can be
obtained by generalizing the treatment done in \cite{murua99ocf}.
Splitting methods with very high order can be constructed as
(\ref{eq:comppsi}) by considering as basic method $S(h)$ a
symmetric method of even order $2q>2$. In that case, it can be
shown that the corresponding number of order conditions is
considerably reduced with respect to (\ref{eq:splittingmethod}).
Thus, for instance, if $S(h)$ is a symmetric splitting method of
order eight, a method (\ref{eq:comppsi}) satisfying the
consistency condition (\ref{eq:consistencygamma}) and the symmetry
condition
\begin{equation}    \label{eq:gammasym}
   \gamma_{s-j+1} = \gamma_j, \qquad 1 \le j \le s
\end{equation}
only needs to satisfy 10 additional conditions to attain order sixteen.

\subsection{Leading term of the local error}
\label{ssect:le}

To construct splitting methods of a given order $p$ within a family of schemes, we choose the number $s$ of stages in such a way that the number of unknowns equals the number of order conditions, so that one typically has a finite number of isolated (real or complex) solutions, each of them leading to a different splitting method. Among them, we will be interested in methods such that, either  $a_i\geq 0$ (and each $b_i$ are arbitrary complex numbers) or  $\Re(a_i)\geq 0$ and $\Re(b_i)\geq 0$. The relevant question at this point is how to choose the `best'
solution in the set of all solutions satisfying the required conditions. It is generally accepted that good splitting methods must have small coefficients $a_i,b_i$. Methods with large coefficients tend to show bad performance in general, which is particularly true when relatively long time-steps $h$ are used. In addition, as mentioned in the introduction, when applying splitting methods to the class of problems considered here, the arguments of the complex coefficients $a_i,b_i$ must also be taken into account.

In order to chose the best scheme among two methods with coefficients of similar size included in sectors with similar angle, we analyze the leading term of the local error of the splitting method. If (\ref{eq:splittingmethod}) is of order $p$, then we formally have that
\begin{eqnarray*}
  \Psi(h) - \e^{h (A+B)} &=& h^{p+1}E_{p+1} + \cO(h^{p+1}),\\
E_{p+1} &:=& \sum_{i_1+\cdots+i_{2m}=p+1} v_{i_1, \ldots,
i_{2m}}(\gamma) \,
 A^{i_1} B^{i_2} \ldots A^{i_{2m-1}} B^{i_{2m}},
\end{eqnarray*}
where $\gamma = (\gamma_1, \ldots, \gamma_{s})$ is given in terms
of the original coefficients $a_i$, $b_i$ of the integrator by
\[
  a_j = \gamma_{j}, \qquad
  b_{j-1} = \frac{\gamma_{j-1}+\gamma_{j}}{2}
\]
($\gamma_0 = \gamma_{s+1} = 0$) and each $v_{i_1, \ldots,
i_{2m}}(\gamma)$ is a linear combination of polynomials in
$\gamma$ \cite{blanes08sac}.

Incorporating that into the results in~\cite{hansen09hos}, it can be shown that, if the smoothness assumption stated in the introduction is increased from $p+1$ to $p+2$, then
for sufficiently small $h$, the local error is dominated by
$||h^{p+1}E_{p+1} u_0||$.

Our strategy to select a suitable method among all possible choices is then the following:
first
choose a subset of solutions with reasonably small maximum norm of
the coefficient vector $(\gamma_1,\ldots,\gamma_{s})$, and then,
among them,  choose the one that minimizes the norm
\begin{equation}
  \label{eq:lenorm}
  \sum_{i_1+\cdots+i_{2m}=p+1} |v_{i_1 \cdots i_{2m}}(\gamma) |
\end{equation}
of the coefficients of the leading term of the local error. This
seems reasonable if one is interested in choosing a splitting
method that works fine for arbitrary operators $A$ and $B$
satisfying the semi-group and smoothness conditions mentioned in
the introduction. Of course, this does not guarantee that a method
with a smaller value of (\ref{eq:lenorm}) will be more precise for
any $A$ and $B$ than another method with a larger value of
(\ref{eq:lenorm}). However, we have observed in practice when
solving the order conditions of different families of splitting
methods, that  the solution that minimizes (\ref{eq:lenorm}) tend
to have smaller values of most (or even all) coefficients $|v_{i_1
\cdots i_{2m}}(\gamma)|$ when compared to a solution having a
larger norm (\ref{eq:lenorm}) of the coefficients of the leading
term of the local error.

When $A$ and $B$ are operators in a real Banach space $X$, then it
makes sense to compute the approximations $u_n=u(t_n)$, as $u_n =
\Re(\Psi(h))u_{n-1}$. In that case, the argument above holds with
$\Psi(h)$ replaced by $\hat \Psi(h) = \Re(\Psi(h))$ and the local
error coefficients $v_{i_1, \ldots, i_{2m}}(\gamma)$ replaced by
$\Re(v_{i_1, \ldots, i_{2m}}(\gamma))$. In that case,
(\ref{eq:lenorm}) should be replaced by
\begin{equation}
  \label{eq:lenormr}
    \sum_{i_1+\cdots+i_{2m}=p+1} |\Re(v_{i_1, \ldots, i_{2m}}(\gamma) )|
\end{equation}
as a general measure of leading term of the local error.

\subsection{High order splitting methods obtained as a composition of simpler methods}
\label{ssect:comp}

\paragraph{Order 6.}

We first consider sixth-order symmetric splitting methods obtained as a composition (\ref{eq:comppsi}) of
the Strang splitting (\ref{eq:strang2}) as basic method. In this case the coefficients $\gamma_i$ must satisfy three order conditions, in addition to the symmetry (\ref{eq:gammasym}) and consistency requirements, to achieve order six. We thus take $s=7$, so that we have three equations and three unknowns. Such a system of polynomial equations has 39 solutions in the complex domain (three real solutions among them), 12 of them giving a splitting method with coefficients of positive real part.
According to the criteria described in Subsection~\ref{ssect:le}, we arrive at the scheme
\begin{eqnarray} \label{P6S7}
\gamma_1 = \gamma_7 &=&  0.116900037554661284389 + 0.043428254616060341762 i\\
\gamma_2 = \gamma_6 &=& 0.12955910128208826275  - 0.12398961218809259330 i,\nonumber \\
\gamma_3 = \gamma_5 &=& 0.18653249281213381780  + 0.00310743071007267534 i,\nonumber \\
\gamma_4 &=& 0.134016736702233270122 + 0.154907853723919152396 i.\nonumber
\end{eqnarray}
This method turns out to correspond to one of those obtained  by Chambers (see Table 4 in \cite{chambers03siw}).

\paragraph{Order 8.}

For consistent symmetric methods (\ref{eq:comppsi}) of order eight, we have seven order conditions. By taking $s=15$ stages, one ends up with a system of seven polynomial equations and seven unknowns. We have performed an extensive numerical search of solutions with small norm, finding 326 complex solutions. Among them, 162 lead to splitting methods whose coefficients possess positive real part. The best method, according to the criteria established in Subsection~\ref{ssect:le}, is
\begin{eqnarray} \label{P8S15}
\gamma_1 = \gamma_{15} &=& 0.053475778387618596606 + 0.006169356340079532510 i, \\
\gamma_2 = \gamma_{14} &=&    0.041276342845804256647 - 0.069948574390707814951 i, \nonumber \\
\gamma_3 = \gamma_{13} &=&    0.086533558604675710289 - 0.023112501636914874384 i, \nonumber \\
\gamma_4 = \gamma_{12} &=&    0.079648855663021043369 + 0.049780495455654338124 i, \nonumber \\
\gamma_5 = \gamma_{11} &=&    0.069981052846323122899 - 0.052623937841590541286 i, \nonumber \\
\gamma_6 = \gamma_{10} &=&    0.087295480759955219242 + 0.010035268644688733950 i,  \nonumber \\
\gamma_7 = \gamma_{9} &=&    0.042812886419632082126 + 0.076059456458843523862 i, \nonumber \\
\gamma_8 &=&    0.077952088945939937643 + 0.007280873939894204350 i. \nonumber
\end{eqnarray}

\paragraph{Order 16.}

Motivated by the results in Section~\ref{sect:aob}, we have also
constructed a splitting method of order 16. Our aim, rather than proposing a very efficient scheme,
is to show that the barrier of order 14 existing for methods
built by applying the recursive composition technique starting from order two (family II) does not apply in general.

The construction procedure can be summarized as follows. We consider
a consistent symmetric composition of the form (\ref{eq:comppsi}), where now the basic method $S(h)$ is any symmetric eighth order scheme. Under such conditions, ten order conditions must be satisfied to achieve order 16. We accordingly take  $s=21$, so that we have ten polynomial equations with ten unknowns. We have performed an extensive numerical search of solutions with relatively small norm, finding $70$ complex solutions with positive real part. Combined with the
$162$ methods of order eight, this leads to $11340$ different $16$-th order splitting methods with $s=315$ stages. Among them, only $324$ give rise to splitting methods with coefficients of positive real part. The coefficients of the method that we have determined as optimal can be found at \texttt{http://www.gicas.uji.es/Research/splitting-complex.html}.

\begin{remark}
Notice that, whereas for the sixth-order method (\ref{P6S7}) one has
$$
\max_{i=1, \ldots,7} \Arg(\gamma_i) - \min_{i=1, \ldots,7}
\Arg(\gamma_i) \approx 1.621 < \frac{\pi}{3}+\frac{\pi}{5}
$$
and then the coefficients $\gamma_i$ are distributed in a narrower
sector than for triple or quadruple jump methods, for the
eighth-order method (\ref{P8S15}) one has
$$
\max_{i=1, \ldots,15} \Arg(\gamma_i) - \min_{i=1, \ldots,15}
\Arg(\gamma_i) \approx 2.5997 >
\frac{\pi}{3}+\frac{\pi}{5}+\frac{\pi}{7}.
$$
This method, whereas being the most efficient, is not the
appropriate one to be used as basic scheme to build higher order methods by composition. The previous
16th-order integrator has been built starting with another 8th-order method whose
coefficients are placed in a narrower sector.
\end{remark}

\section{Numerical tests} \label{sect:nt}

For our numerical experiments, we consider two different test problems: a linear reaction-diffusion equation, and the semi-linear Fisher's equation, both with periodic boundary conditions in space. It should be mentioned here that this last case is not covered by the theoretical framework summarized in the Introduction. For each case, we detail the experimental setting
and collect the results achieved by the different schemes. Our main purpose here is just to illustrate the
performance of the new splitting methods to carry out the time integration as compared with those constructed by
using the Yoshida--Suzuki triple jump
composition technique for both examples. Notice that, in this sense, the particular
scheme used to discretize in space is irrelevant.  For that reason, and to keep the
treatment as simple as possible,
we have applied a simple second-order finite difference scheme in space.

The numerical approximations $u_n$ obtained by each method $\Psi(h)$ are computed as
$u_n = \Re(\Psi(h)) u_{n-1}$. In other words, we project on the real axis after completing each time step.

\subsection{A linear parabolic equation}
Our first test-problem is the  scalar equation in one-dimension
\begin{eqnarray} \label{eq:lrd}
\frac{\partial u(x,t)}{\partial t} = \alpha \Delta u(x,t) + V(x,t) u(x,t),  \qquad u(x,0) = u_0(x),
\end{eqnarray}
with $u_0(x) = \sin(2 \pi x)$ and periodic boundary conditions in
the space domain $[0,1]$. We take $\alpha = \frac{1}{4}$,
$V(x,t)=3+\sin(2 \pi x)$ and discretize in space
$$x_j= j (\delta x), \qquad j=1,\ldots,N \quad \mbox{ with } \quad \delta x = 1/N,
$$
thus arriving at the differential equation
\begin{equation} \label{eq:problem0}
\frac{dU}{dt} = \alpha A U + B U,
\end{equation}
where $U=(u_1, \ldots,u_N) \in \R^N$. The Laplacian $\Delta$ has been approximated by the matrix $A$ of size $N\times N$ given by
$$
A=\frac{1}{(\delta x)^2}\begin{pmatrix}
-2&1& & & 1\\
1 & -2 & 1 \\
 & 1 & -2 & 1 \\
 & & \ddots & \ddots & \ddots \\
1 & & & 1 & -2\end{pmatrix},
$$
and  $B=\mbox{diag}(V(x_1), \ldots, V(x_N) )$. We take
$N=100$ points and compare different composition methods by
computing the corresponding approximate solution on the time interval $[0,1]$.
In  particular, we consider the following schemes:
\begin{itemize}
\item {\bf Strang}: The second-order symmetric Strang splitting method (\ref{eq:strang2});
\item {\bf (TJ6)}: The sixth-order triple jump method (Proposition 2.1 in \cite{castella09smw})  based on Strang's second-order method;
\item {\bf (TJ6A)}: The sixth-order triple jump method  (Proposition 2.2 in \cite{castella09smw}) based on Strang's second-order method;
\item {\bf (TJ8A)}: The eighth-order triple jump method  (Proposition 2.2 in \cite{castella09smw})  based on Strang's second-order method;
\item {\bf (P6S7)}: The sixth-order method (\ref{P6S7});
 \item {\bf (P8S15)}: The eighth-order method (\ref{P8S15}).
\end{itemize}
We compute the error of the numerical solution at time $t=1$ (in
the $2$-norm) as a function of the number of evaluations of the
basic method (the Strang splitting) and represent the outcome in
Figure \ref{fig:reactdiff}. In the left panel we collect the
results achieved by the Strang splitting and the previous
sixth-order composition methods, whereas the right panel
corresponds to eighth-order methods. We have also included, for
reference, the curve obtained by (P6S7).

The relative cost (w.r.t. Strang) of a method composed of $s$ steps is approximated by $4 s$, where the factor $4$ stands here for an average ratio between the cost of complex arithmetic compared to real arithmetic. A remarkable outcome of these experiment is that methods (P6S7) and  (P8S15) outperform Strang's splitting (and actually all other methods tested here) even for low tolerances. Scheme (P8S15), in particular, proves
to be the most efficient in the whole range explored. The gain with respect to triple jump methods is also  significant and completely support the approach followed here.  \\

\begin{figure}[th!]
\begin{center}
\scalebox{0.52}{\includegraphics{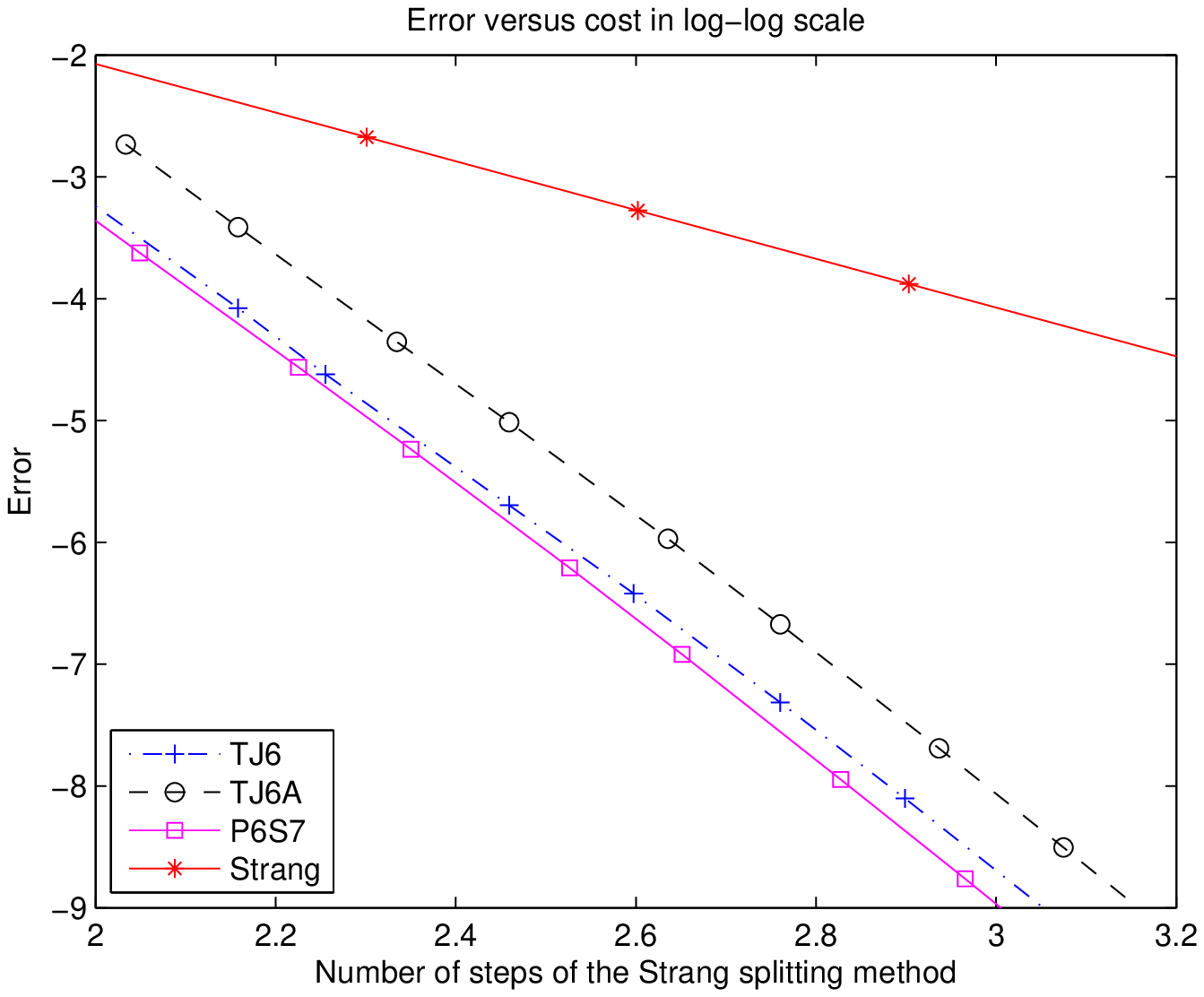}}
\scalebox{0.52}{\includegraphics{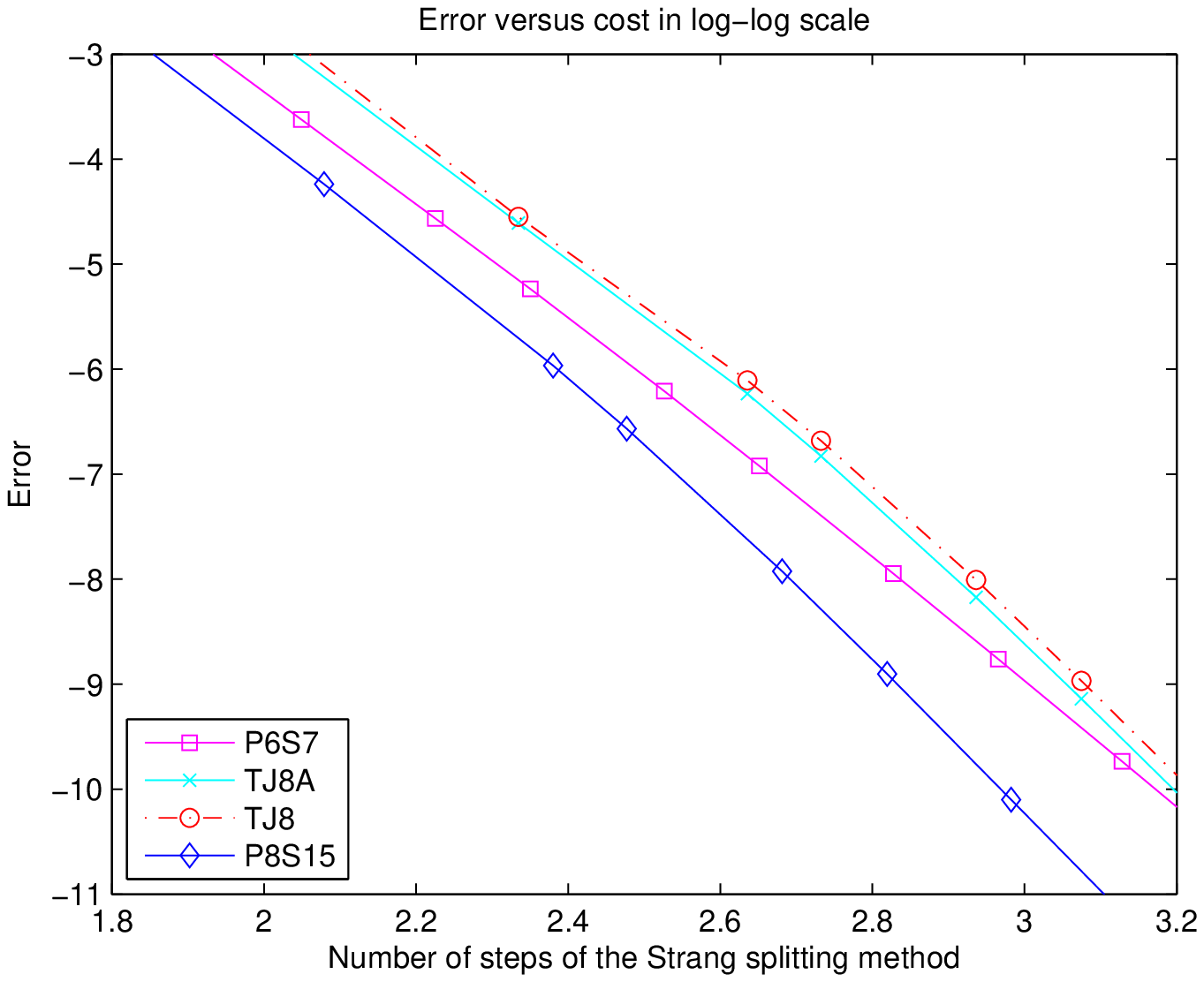}}
\end{center}
\caption{{\small Error versus number of steps for the linear  reaction-diffusion equation (\ref{eq:lrd}).}}
\label{fig:reactdiff}
\end{figure}

\subsection{The semi-linear reaction-diffusion equation of Fisher}
Our second test-problem is  the scalar equation in one-dimension
\begin{eqnarray} \label{eq:nlrd}
\frac{\partial u(x,t)}{\partial t}  = \Delta u(x,t) + F(u(x,t)), \qquad u(x,0) = u_0(x),
\end{eqnarray}
with periodic boundary conditions in the space domain $[0,1]$. We take, in particular,
Fisher's potential $$F(u)=u(1-u).$$
The splitting considered here corresponds to solving, on the one hand,  the linear equation with the
operator $A$ being the Laplacian, and on the other hand, the nonlinear ordinary differential equation
$$
\frac{\partial u(x,t)}{\partial t} = u(x,t)(1-u(x,t))
$$
with initial condition
$$
u(x,0)=u_0(x).
$$
Note that it can be solved analytically as
$$
u(x,t)) = u_0(x)+ u_0(x)(1-u_0(x))\frac{(\e^t-1)}{1+u_0(x) (\e^t-1)},
$$
which is well defined for small complex time $t$. We proceed in the same way as for the previous
linear case, starting with $u_0(x) = \sin(2 \pi x)$.
After discretization in space,
we arrive at the differential equation
\begin{equation} \label{eq:problem1}
\frac{dU}{dt} = A U + F(U),
\end{equation}
where  $U=(u_1, \ldots, u_N) \in \R^N$ and  $F(U)$ is now defined by
$$
F(U)=\big(u_1(1-u_1), \ldots, u_N(1-u_N) \big).
$$
We choose $N=100$ and compute the error (in the $2$-norm) at the
final time $t=1$ by applying the same composition methods as in
the linear case. The results are collected in Figure
\ref{fig:reactdiff-2}, where identical notation has been used.
Notice that, strictly speaking, the theoretical framework upon
which our strategy is based does not cover this nonlinear problem.
Nevertheless, the results achieved are largely similar to the
linear case. In particular, the new 8th-order composition method
is the most efficient even for moderate tolerances.

\begin{figure}[h!]
\begin{center}
\scalebox{0.52}{\includegraphics{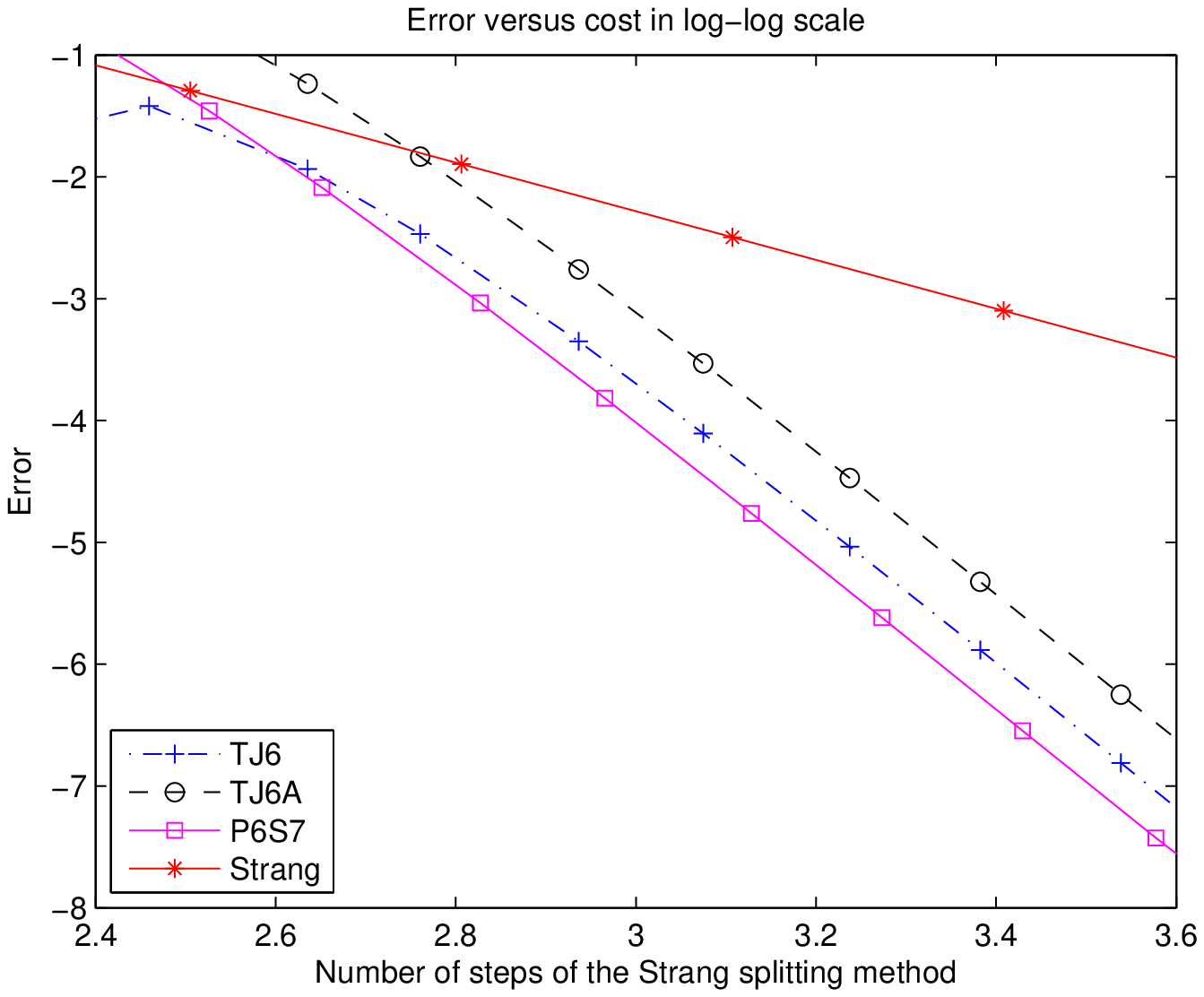}}
\scalebox{0.52}{\includegraphics{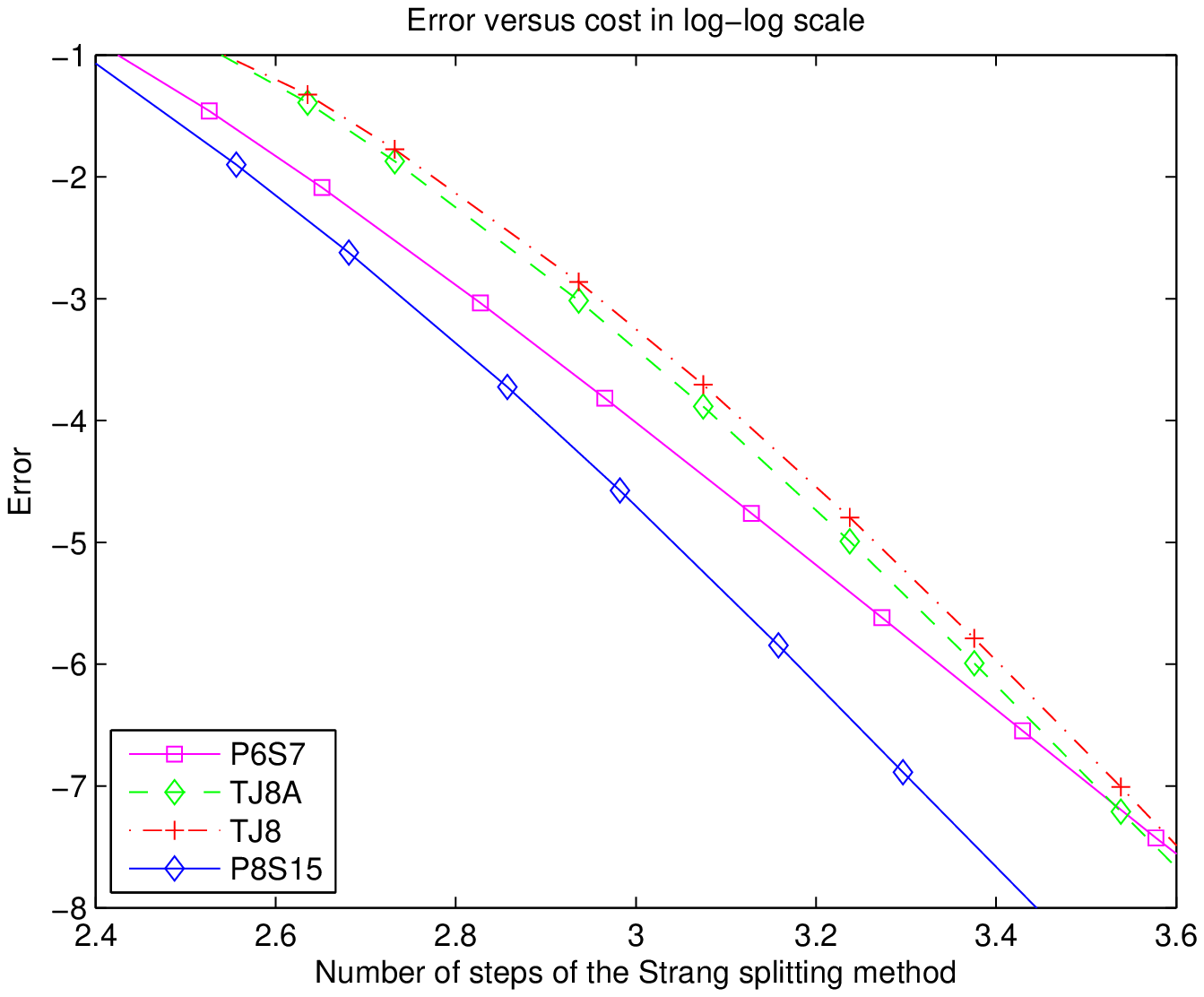}}
\end{center}
\caption{{\small Error versus number of steps for the semi-linear reaction-diffusion equation
(\ref{eq:nlrd}).}}
\label{fig:reactdiff-2}
\end{figure}

\section{Concluding remarks}
\label{sec.6}

Splitting methods with real coefficients for the numerical integration of differential equations of
order higher than two have necessarily some negative coefficients. This feature does not
suppose any special impediment when the differential equation evolves in a group, but may be
unacceptable when it is defined in a semi-group, as is the case with the evolution partial
differential equations considered in this paper. One way to get around this fundamental difficulty
is to consider splitting schemes with complex coefficients having positive real part. This has been recently proposed for diffusion equations in \cite{castella09smw,hansen09hos}. Splitting and composition methods with complex coefficients have been considered in different contexts in the literature (see \cite{blanes10smw} and references therein).

In \cite{castella09smw,hansen09hos}, splitting methods up to order 14 with
complex coefficients with non-negative real part have been recursively constructed either by the so-called triple-jump compositions or by the quadruple-jump compositions, starting from the symmetric second-order Strang splitting.
In this work we  prove that there exists indeed an order barrier of 14 for methods constructed in this way. More generally, we show that no method
of order higher than 14 with coefficients having non-negative real part can be constructed by sequential $s$-jump compositions starting from a symmetric method of order 2. We further show, by explicitly obtaining
methods of order 16 (as the composition of a basic symmetric method of order 8), that this order barrier does not apply for general composition methods (non-necessarily constructed by recursive applications of $s$-jump compositions) with complex coefficients with non-negative real part.

In addition to this order barrier, another drawback of methods
resulting from applying the $s$-jump composition procedure is that
for high orders they require a larger number of stages (i.e. number
of compositions of the basic symmetric second order method) than
methods obtained by directly solving the order conditions with the
minimal number of stages. For instance, methods of order 6
(respectively, 8) obtained with triple jump compositions need 9
(resp. 27) compositions of the basic second order method, whereas,
as we show in the present work, methods of order 6 (resp. 8) can
be constructed (by directly solving for the required order
conditions) with 7 (resp. 15) stages. An analysis of the local
error coefficients supported by numerical tests shows that the
methods proposed here are more efficient than those obtained
in~\cite{castella09smw,hansen09hos} by applying the recursive
triple jump and quadruple jump constructions. An additional
requirement when choosing a given method is that the arguments of
the complex coefficients of the scheme have also to be taken into
account. This constitutes a critical point for evolution equations
where one of the operators (say, $A$) has non-real eigenvalues in
the right-hand side of the complex plane, as occurs, in
particular, with the complex Ginzburg--Landau equation
(\ref{eq:landau}). In such a case, splitting methods of the form
(\ref{eq:splittingmethod}) where one of the two sets of
coefficients $a_i$'s or $b_i$'s is entirely contained in the
positive real axis, whereas the other set is included in the
right-hand side of the complex planes are particularly well
suited. Such splitting methods cannot be constructed as
composition methods with the Strang splitting as basic method, so
that a separate study is required to get the most efficient
schemes within this class.

Based on the theoretical framework worked out in \cite{hansen08esf}, the integrators proposed here can be applied to the numerical integration of linear evolution equations involving unbounded operators in an infinite dimensional space, like linear diffusion equations. As a matter of fact, although the theory developed in
\cite{hansen08esf} does not cover the generalization to semi-linear evolution equations,
 we have also included in our numerical tests a system of ODEs obtained from semi-linear evolution equations with a certain space discretization. All the numerical tests carried out with periodic boundary conditions show a considerable improvement in efficiency of our new methods with respect to existing splitting schemes. 
 A remarkable feature of the
 new eighth-order composition method when applied to both the linear and semi-linear diffusion examples 
 is that it is more efficient than all the other integrators of order $p\leq 8$ in the whole range of tolerances
 explored when periodic boundary conditions are considered.
 
 Concerning other (e.g. Dirichlet and Neumann) boundary conditions, the experiments carried out in  
 \cite{hansen09hos} for linear problems
 with methods obtained by applying the triple and quadruple jump technique show 
 the existence of an order reduction phenomenon. Its origin is attributed by the authors of 
  \cite{hansen09hos} to the fact that Condition 2 in the introduction is not generally satisfied in this setting.
  In other words, terms of the form $E_{p+1} \, \e^{t(A+B)}$ in (\ref{eq:formalLE}) are not uniformly bounded
  on the interval $[0,T]$ for some $T > 0$. As a consequence, 
the classical convergence order is no longer guaranteed in that case. 
This order reduction is also present, of course, 
when the methods introduced in this paper are applied to linear but also semi-linear problems. Nevertheless,
as has been pointed out in  \cite{hansen09hos}, splitting schemes of high order involving complex coefficients
may still be advantageous in comparison with, say, Strang splitting when applied to linear 
parabolic problems with
Dirichlet or Neumann boundary conditions, as they often lead to smaller errors and the order reduction is
somehow confined to a neighborhood of the boundary. The situation, in our opinion, calls for a detailed study of the
effect of boundary conditions other than the periodic case in the global efficiency of splitting methods
with complex coefficients and the analysis of their applicability for more general parabolic problems than those
considered in this paper.

\subsection*{Acknowledgements}

The work of SB, FC and AM has been partially supported by Ministerio de
Ciencia e Innovaci\'on (Spain) under the coordinated project MTM2010-18246-C03
(co-financed by FEDER Funds of the European Union). Financial support from the
``Acci\'on Integrada entre Espa\~na y Francia'' HF2008-0105 is also acknowledged.
AM is additionally funded by project EHU08/43 (Universidad del Pa\'{\i}s Vasco/Euskal
Herriko Unibertsitatea).

\bibliographystyle{alpha}

\end{document}